\documentclass[12pt,reqno]{amsart}
\usepackage{amssymb,amsmath,graphicx,
	amsfonts,euscript}
\usepackage{color}
\usepackage{mathrsfs}

\def\nn{\nonumber}
\def\div{ \hbox{\rm div}\,  }
\theoremstyle{plain}

\newtheorem{theorem}{Theorem}[section]
\newtheorem{lemma}[theorem]{Lemma}

\newtheorem{remark}[theorem]{Remark}
\numberwithin{equation}{section}

\newcommand{\norma}[2]{\left\lVert #1 \right\rVert_{#2}}
\newcommand{\ta}{\theta}
\newcommand{\Ga}{\Gamma}
\newcommand{\dt}{\partial_t}
\newcommand{\dx}{\partial_x}
\newcommand{\dy}{\partial_y}
\newcommand{\p}{{\mathbb P}}
\newcommand{\q}{{\mathbb Q}}

\newcommand{\jap}[1]{\left\langle #1 \right\rangle}
\newcommand{\norm}[1]{\left\lVert #1 \right\rVert}
\newcommand{\R}{\mathbb{R}}
\newcommand{\T}{\mathbb{T}}

\begin{document}

\title[stability analysis of the compressible Euler equations ]{Linear stability analysis of the  Couette flow for the two dimensional  non-isentropic compressible Euler equations}

\author[Xiaoping Zhai]{ Xiaoping Zhai}

\address{ School  of Mathematics and Statistics, Shenzhen University,
 Shenzhen, 518060, China}
\email{zhaixp@szu.edu.cn}

\begin{abstract}
This note is devoted to the linear stability  of the  Couette flow for the   non-isentropic compressible Euler equations in a domain $\mathbb{T}\times \mathbb{R}$. Exploiting
the several conservation laws originated from  the special structure of the linear system,
we obtain a Lyapunov type instability for the  density, the temperature, the compressible part of the velocity field, and also obtain  an inviscid damping for the incompressible part of the velocity field.
	\end{abstract}
\maketitle

\section{ Introduction and the main result}
Inspired by \cite{dole} and \cite{doln}, in the present paper, we are interested in the long-time asymptotic behaviour of the linearized two dimensional  non-isentropic compressible Euler equations in a domain $\mathbb{T}\times \mathbb{R}$. The governing equations (in non-dimensional variables) are
\begin{align}\label{m1}
\frac{\partial{\varrho}}{\partial  t}+{\mathbf{u}}\cdot\nabla {\varrho}+{\varrho}\div{\mathbf{u}}=0,
\end{align}
\begin{align}\label{m2}
{{\varrho}}\left[\frac{\partial{\mathbf{u}}}{\partial  t}+{\mathbf{u}}\cdot\nabla {\mathbf{u}}\right]
=-\frac{\nabla P}{\gamma  {M}^2} ,
\end{align}
\begin{align}\label{m3}
P={\varrho}\vartheta,
\end{align}
\begin{align}\label{m4}
{{\varrho}}\left[\frac{\partial\vartheta}{\partial  t}+{\mathbf{u}}\cdot\nabla \vartheta\right]
=- (\gamma-1)P\div {\mathbf{u}},
\end{align}
where {for }  $(x,y)\in\mathbb{T}\times \mathbb{R}$ and $\mathbb{T}=\mathbb{R}/\mathbb{Z}$,
 ${\mathbf{u}}$ is the velocity vector, ${{\varrho}}$ is the density, $P$ is the pressure, $\vartheta$ is the
temperature, $\gamma$ is the ratio of specific heats, and $M$ is the Mach number of the reference state.
The question of stability of Couette flows has a long history, see \cite{dole}, \cite{doln}, \cite{Q12+1}, \cite{Q13}, \cite{Q14}, \cite{Q17},  \cite{Q18}, \cite{Q19},  \cite{Q21} for the compressible fluids and see \cite{Q1}, \cite{Q11},  \cite{Q12}, \cite{Q15}, \cite{Q16}, \cite{Q20}, \cite{Q22} for the  incompressible fluids.

The Couette flow,
\begin{align}\label{}
{\mathbf{u}}_{sh}=\begin{pmatrix}
				y\\ 0
			\end{pmatrix},\quad\varrho_{sh}=1,\quad \vartheta_{sh}=1,
\end{align}
is clearly a stationary solution of \eqref{m1}--\eqref{m4}.
 Our goal is to understand the stability and large-time behavior of perturbations near the Couette flow.

Denote
\begin{equation*}
		{\rho}={\varrho}-\varrho_{sh},  \qquad{\ta}=\vartheta-\vartheta_{sh},\qquad
		{\mathbf{v}}={\mathbf{u}}-{\mathbf{u}}_{sh}.
	\end{equation*}
The linearized system around the homogeneous Couette flow read as follows
\begin{eqnarray}\label{m5}
\left\{\begin{aligned}
&\dt {\rho} +y\dx{\rho}+\div {\mathbf{v}}=0,\\
&\dt {\mathbf{v}}+y\dx {\mathbf{v}}+\begin{pmatrix}
				v^y\\ 0
			\end{pmatrix}+\frac{1}{\gamma {M}^2} (\nabla{\rho}+\nabla{\ta})=0,\\
&\dt {\ta}+y\dx {\ta}+(\gamma-1)\div{\mathbf{v}}=0.
\end{aligned}\right.
\end{eqnarray}
Define
\begin{equation*}
		\alpha=\div {\mathbf{v}}, \qquad \omega=\nabla^\perp\cdot {\mathbf{v}}, \quad \hbox{with\quad $\nabla^\perp=(-\dy,\dx)^T$},
	\end{equation*}
according to
	the Helmholtz projection operators,  we have
	\begin{align}\label{m6}
{\mathbf{v}}=(v^x,v^y)^T\stackrel{\mathrm{def}}{=}\p [{\mathbf{v}}]+\q [{\mathbf{v}}]
\end{align}
 with
 \begin{align}\label{m7}
\p [{\mathbf{v}}]\stackrel{\mathrm{def}}{=}\nabla^\perp \Delta^{-1}\omega, \qquad\q [{\mathbf{v}}]\stackrel{\mathrm{def}}{=}\nabla \Delta^{-1}\alpha.
\end{align}
From the above definition, one can infer that
\begin{equation}\label{m9}
		v^y=\dy(\Delta^{-1})\alpha+\dx(\Delta^{-1})\omega,
	\end{equation}
hence, we  can rewrite \eqref{m5}  in terms of $(\rho,\alpha,\omega,\ta)$ that
\begin{eqnarray}\label{m10}
\left\{\begin{aligned}
&\dt \rho +y\dx\rho+\alpha=0,\\
&\dt \alpha+y\dx \alpha+2\dx (\dy(\Delta^{-1})\alpha+\dx(\Delta^{-1})\omega)+\frac{1}{\gamma{M}^2}(\Delta\rho+\Delta\theta)=0,\\
&	\dt \omega+y\dx \omega-\alpha=0,\\	
&\dt \ta +y\dx\ta+(\gamma-1)\alpha=0.
\end{aligned}\right.
\end{eqnarray}
Obviously, the above system \eqref{m10} is a closed system regarding of $(\rho,\alpha,\omega,\ta)$.

Now, we can state the main result of the present paper.

 \begin{theorem}\label{dingli}
		      	 Let $(\rho^{in}, \ \omega^{in}, \ \theta^{in})\in H^1_x(\T)H^2_y(\R)$ and $\alpha^{in}\in H^{-\frac{1}{2}}_x(\T)L^2_y(\R)$ be the initial data of \eqref{m10} with
$$\int_\mathbb{T} \rho_{in} \,dx=\int_\mathbb{T} \omega_{in} \,dx=\int_\mathbb{T} \alpha_{in} \,dx=\int_\mathbb{T} \ta_{in} \,dx=0.$$
 Then
		\begin{align}\label{you17}
			\norma{\p[{\mathbf{v}}]^x(t)}{L^2}\le&\frac{C{M}}{\langle t\rangle^{\frac{1}{2}}}\bigg(\norma{ \frac{\rho^{in}+\ta^{in}}{\gamma{M}}}{H^{-\frac12}_xL^2_y}+\norma{\alpha^{in}}{H^{-\frac12}_xH^{-1}_y} 
+\norma{ \frac{\rho^{in}+\ta^{in}+\gamma\omega^{in}}{\gamma}}{H^{-\frac12}_xH^{\frac12}_y}\bigg)\nonumber\\
			&+\frac{C}{\langle t\rangle}\norma{\frac{\rho^{in}+\ta^{in}+\gamma\omega^{in}}{\gamma}}{H^{-1}_xH^1_y},
		\end{align}		
\begin{align}\label{you18}
			\norma{\p[{\mathbf{v}}]^y(t)}{L^2}
			\le&\frac{C{M}}{\langle t\rangle^{\frac32}}\bigg(\norma{ \frac{\rho^{in}+\ta^{in}}{\gamma{M}}}{H^{-\frac12}_xH^{1}_y}+\norma{\alpha^{in}}{H^{-\frac12}_xL^2_y}+\norma{\frac{\rho^{in}+\ta^{in}+\gamma\omega^{in}}{\gamma}}{H^{-\frac12}_xH^{\frac{3}{2}}_y}\bigg)\nonumber\\
			&+\frac{C}{\langle t\rangle^2}\norma{\frac{\rho^{in}+\ta^{in}+\gamma\omega^{in}}{\gamma}}{H^{-1}_xH^2_y}.
		\end{align}
	Moreover,
      	\begin{align}\label{you16}	&\norma{Q[{\mathbf{v}}](t)}{L^2}+\frac{\gamma}{{M}}\norma{\rho(t)}{L^2}+\frac{\gamma}{{M}}\norma{\ta(t)}{L^2}\nonumber\\
&\quad\le C \jap{t}^\frac12\Bigg({(\gamma-1)}\left\|\frac{(\gamma-1)\rho^{in}-\ta^{in}}{{M}}\right\|_{{L^2}}\nonumber\\
&\qquad\qquad+\left(\norm{\frac{\rho^{in}+\ta^{in}}{\gamma M}}_{L^2}+\norm{\alpha^{in}}_{H^{-1}}
+\norm{\frac{\rho^{in}+\ta^{in}+\gamma\omega^{in}}{\gamma}}_{H^{\frac12}}\right)
\Bigg).
			\end{align}
\begin{remark}
In a forthcoming paper, we consider the linear stability  of the  Couette flow for the non-isentropic compressible Euler equations in three dimensional.
\end{remark}		     
 \end{theorem}
\begin{remark}
Our stability analysis coincides with the isentropic compressible Euler equations discussed in \cite{dole} and \cite{doln} if  we neglect the effect of the temperature. Moreover,  following a similar argument as \cite{dole} and \cite{doln},   we can also obtain the lower bound for the density, the temperature and the compressible part of the velocity field as
\begin{align*}
\big\|\q[{\mathbf{v}}]\big\|_{L^2}+\frac{1}{M}\norma{\rho+\ta}{L^2}\geq \langle t \rangle^{\frac12}  C_{in}
\end{align*}	
if the initial data up to a nowhere dense set.
\end{remark}

\section{The proof of the main theorem}
First of all, we are concerned with the dynamics of the $x$-averages of the perturbations.
In order to
reveal the distinction between
the zero mode case $k=0$ and the nonzero modes $k\not =0$. We define
\begin{align}
f_0(y)\stackrel{\mathrm{def}}{=}\frac{1}{2\pi}\int_{\mathbb T}f(x,y)  \,dx, \qquad f_{\not =}(x,y) \stackrel{\mathrm{def}}{=}f(x,y) -f_0(y),
\end{align}
which represents the projection onto $0$ frequency and the projection onto non-zero frequencies.

Let
	\begin{align*}
		&\widehat{f}(k,\eta)=\frac{1}{2\pi}\iint_{\T\times\mathbb{R}}e^{-i(kx+\eta y)}f(x,y)\,dxdy,
	\end{align*}
	then we define  $f\in H^{s_1}_xH^{s_2}_y$ if
	\begin{equation*}
		\norm{f}_{H^{s_1}_xH^{s_2}_y}^2=\sum_k\int \langle k\rangle^{2s_1}\langle \eta \rangle^{2s_2} |\hat{f}|^2(k,\eta)\,d\eta< +\infty.
	\end{equation*}
	Moreover, we also denote the usual $H^s(\T\times \R)$ space as
	\begin{equation*}
		\norm{f}_{H^{s}}^2=\sum_k\int \langle k,\eta\rangle^{2s} |\hat{f}|^2(k,\eta)\,d\eta.
	\end{equation*}

Integration in $x$ equations in \eqref{m10}, one infer that
\begin{eqnarray}\label{ping1234}
\left\{\begin{aligned}
    &\dt \rho_0=-\alpha_0,\\
		&\dt \alpha_0=-\frac{1}{\gamma{M}^2}{\partial_{yy}} \rho_0-\frac{1}{\gamma{M}^2}{\partial_{yy}} \ta_0,\\
		&\dt \omega_0=\alpha_0,\\
&\dt \ta_0=-(\gamma-1)\alpha_0.
\end{aligned}\right.
\end{eqnarray}
From the above equation \eqref{ping1234}, we can further get $\alpha_0, \rho_0+\ta_0$ satisfy the following wave equations:
\begin{equation}\label{ping2+1}
		\partial_{tt} \alpha_0-\frac{1}{{M}^2}{\partial_{yy}}\alpha_0=0, \qquad \text{in }\R,
	\end{equation}
\begin{equation}\label{ping3+1}
		\partial_{tt} (\rho_0+\ta_0)-\frac{1}{{M}^2}{\partial_{yy}}(\rho_0+\ta_0)=0, \qquad \text{in }\R.
	\end{equation}
Hence, given $\rho_0^{in}=\alpha_0^{in}=\ta_0^{in}=\omega_0^{in}=0$, according to the explicit representation formula for
\eqref{ping2+1}, \eqref{ping3+1}, we can get for all $t\ge0$
$$\rho_0(t)=\alpha_0(t)=\ta_0(t)=\omega_0(t)=0.$$

Consequently, in our analysis we can decouple the evolution of the $k = 0$ mode from the rest of the perturbation.
Let us consider the following coordinate transform
\begin{align*}
\left(
  \begin{array}{c}
    x \\
    y \\
  \end{array}
\right)
\mapsto
\left(
  \begin{array}{c}
    X \\
    Y \\
  \end{array}
  \right)=
  \left(
  \begin{array}{c}
    x-yt \\
    y \\
  \end{array}
  \right).
\end{align*}
Under the new coordinate transform,
 the differential operators change as follows
	\begin{align*}
			\dx={\partial_X},\quad \dy = {\partial_Y}-t{\partial_X},\quad
			\Delta= \Delta_L\stackrel{\mathrm{def}}{=}{\partial_{XX}}+({\partial_Y}-t{\partial_X})^2.
	\end{align*}
Define
\begin{align*}
R(t,X,Y)&=\rho(t,X+tY,Y),\quad
			A(t,X,Y)=\alpha(t,X+tY,Y),\\
			\Omega(t,X,Y)&=\omega(t,X+tY,Y),\quad
\Theta(t,X,Y)=\ta(t,X+tY,Y).
\end{align*}
Then, the linear system \eqref{m10} reduces to the following system in the new coordinates
\begin{eqnarray}\label{hao}
\left\{\begin{aligned}
&\dt R=-A,\\
&\dt A=-2{\partial_X}({\partial_Y}-t{\partial_X})(\Delta_L^{-1})A-2{\partial_{XX}}(\Delta_L^{-1})\Omega-\frac{1}{\gamma{M}^2}(\Delta_LR+\Delta_L{\Theta}),\\
&\dt \Omega=A,\\
&\dt {\Theta}=-(\gamma-1)A.
\end{aligned}\right.
\end{eqnarray}

From \eqref{m10}, we also know that
\begin{eqnarray}\label{you}
\left\{\begin{aligned}
&(\dt+y\dx) (\rho+\omega)=0,\\
& (\dt+y\dx) (\ta+(\gamma-1)\omega)=0.
\end{aligned}\right.
\end{eqnarray}
In particular, \eqref{you} implies that  $\rho+\omega$ and $\ta+(\gamma-1)\omega$ are transported by the Couette flow.
Hence, if we further define
\begin{align}\label{you1}
{\beta}(t,X,Y)\stackrel{\mathrm{def}}{=}R(t,X,Y)+\Omega(t,X,Y), \quad \Gamma(t,X,Y)\stackrel{\mathrm{def}}{=}{\Theta}(t,X,Y)+(\gamma-1)\Omega(t,X,Y),
\end{align}
then, we have
\begin{eqnarray*}
\left\{\begin{aligned}
&\dt {\beta}=0,\\
& \dt \Gamma=0,
\end{aligned}\right.
\end{eqnarray*}
which implies that
\begin{align}\label{you3}
{\beta}={\beta}^{in}=\rho^{in}+\omega^{in},\quad \Ga=\Ga^{in}=\theta^{in}+(\gamma-1)\omega^{in}.
\end{align}

Moreover,
one can infer from \eqref{you1} and \eqref{you3} that
\begin{eqnarray*}
\left\{\begin{aligned}
&{\Omega(t,X,Y)}={{\beta}^{in}(X,Y)-R(t,X,Y)},\\
& \Omega(t,X,Y)=\frac{\Ga^{in}(X,Y)-{\Theta}(t,X,Y)}{\gamma-1},
\end{aligned}\right.
\end{eqnarray*}
which gives
\begin{align}\label{you5}
\Omega(t,X,Y)=\frac{{\beta}^{in}(X,Y)+\Ga^{in}(X,Y)}{\gamma}-\frac{R(t,X,Y)+{\Theta}(t,X,Y)}{\gamma}.
\end{align}
To accelerate the proof, we continue to introduce the following notation
\begin{align}\label{you6}
\delta=\frac{R+{\Theta}}{\gamma}.
\end{align}

In view of \eqref{m9},  \eqref{you5} and \eqref{you6}, we get
\begin{align*}
		V^y=&({\partial_Y}-t{\partial_X})\Delta_L^{-1}A+{\partial_X}\Delta_L^{-1}\Omega \\
		=&({\partial_Y}-t{\partial_X})\Delta_L^{-1}A+\frac1\gamma{\partial_X}\Delta_L^{-1}({\beta}^{in}+\Ga^{in})-\frac1\gamma{\partial_X}\Delta_L^{-1}(R+{\Theta})\\
=&({\partial_Y}-t{\partial_X})\Delta_L^{-1}A+\frac1\gamma{\partial_X}\Delta_L^{-1}({\beta}^{in}+\Ga^{in})-{\partial_X}\Delta_L^{-1}\delta.
	\end{align*}

As a result, \eqref{hao} reduces to the following system
\begin{eqnarray}\label{haohao}
\left\{\begin{aligned}
&\dt \delta=-A,\\
&\dt A=-2{\partial_X}({\partial_Y}-t{\partial_X})(\Delta_L^{-1})A\\ &\qquad+\left(-\frac{1}{{M}^2}\Delta_L+2{\partial_{XX}}(\Delta_L^{-1})\right)\delta-\frac{2}{\gamma}{\partial_{XX}}(\Delta_L^{-1})({\beta}^{in}+\Ga^{in}).
\end{aligned}\right.
\end{eqnarray}
Compared to  \eqref{hao}, the above system \eqref{haohao} is a closed $2\times 2$ system only involving $\delta$ and $A$.
Moreover, the above system \eqref{haohao} is almost the same as the isentropic compressible Euler system discussed in \cite{dole}, \cite{doln}. Hence, most of the following argument can be founded in \cite{dole}, \cite{doln}, we present them here for reader's convenience.

We define the symbol associated to $-\Delta_L$ as
	\begin{align*}
		p(t,k,\eta)&=k^2+(\eta-kt)^2,
	\end{align*}
and denote the symbol associated to the operator  $2{\partial_X}({\partial_Y}-t{\partial_X})$
	as
	\begin{align*}
		(\dt p)(t,k,\eta)=-2k(\eta-kt).
	\end{align*}
	
Taking the Fourier transform in \eqref{haohao}, we have		
\begin{eqnarray}\label{you9}
\left\{\begin{aligned}
&\dt {\widehat{\delta}}=-{\widehat{A}},\\
&\dt {\widehat{A}}=\frac{\dt p}{ p}{\widehat{A}}+\bigg(\frac{ p}{{M}^2}+\frac{2k^2}{ p}\bigg){\widehat{\delta}}-\frac{2k^2}{ \gamma p}(\widehat{\beta}^{in}+\widehat{\Gamma}^{in}).
\end{aligned}\right.
\end{eqnarray}
Motivated by \cite{dole}, \cite{doln}, we introduce the following metric

\begin{align*}
Z(t)=
\begin{pmatrix}
			Z_1(t) \\
\\
 Z_2(t)
		\end{pmatrix}
=\begin{pmatrix}
			\frac{1}{{M} p^{\frac14}}{\widehat{\delta}}(t) \\
\\
 \frac{1}{ p^{\frac34}}{\widehat{A}}(t)
		\end{pmatrix}.
\end{align*}
	By a direct computation we find that $Z(t)$ satisfy
	\begin{eqnarray}\label{you11}
\left\{\begin{aligned}
&{\frac{{\mathrm{d}}}{{\mathrm{d}} t}} Z(t)=L(t)Z(t)+F(t)(\widehat{\beta}^{in}+\widehat{\Gamma}^{in}),\\
 &Z(0)=Z^{in}
\end{aligned}\right.
\end{eqnarray}

	where
	\begin{equation}\label{you12}
		L(t)= \begin{bmatrix}
			\displaystyle -\frac{\dt p}{4p} & \displaystyle -\frac{\sqrt{p}}{{M}} \\
			\displaystyle \frac{2\sqrt{p}}{{M}} +\frac{2{M} k^2}{p^{3/2}}&\displaystyle \frac{\dt p}{4p}
		\end{bmatrix},\ \ \ F(t)=\begin{pmatrix}
			0 \\\displaystyle -\frac{2k^2}{\gamma p^{7/4}}
		\end{pmatrix}
	\end{equation}
	and
	\begin{equation}\label{you13}
		Z^{in}=\left(\frac{1}{M(k^2+\eta^2)^{\frac14}}
{\widehat{\delta}}^{in},\frac{1}{(k^2+\eta^2)^{\frac34}}{\widehat{A}}^{in}\right)^T.
	\end{equation}

Observe that if we are able to get a uniform bound on $|Z|$, in view of the weight on $\delta$ and $A$, we can obtain the desired time decay. This point can be illustrated by the following lemma.
\begin{lemma}\label{you14}
Let $p=-\widehat{\Delta}_L=k^2+(\eta-kt)^2$, then for any function $f\in H^{s+2\beta}(\mathbb{T}\times\mathbb{R})$, it holds that
	\begin{equation}\label{you15}
	\norma{p^{-\beta}f}{H^s}\le C \frac{1}{\langle t \rangle^{2\beta}}\norma{f}{{H^{s+2\beta}}}\qquad \norma{p^{\beta} f}{H^s}\le C \langle t \rangle^{2\beta} \norma{f}{{H^{s+2\beta}}},
	\end{equation}
	for any $\beta>0$.
\end{lemma}
\begin{proof}
	The bounds \eqref{you15} follows just by Plancherel Theorem and the basic inequalities for japanese brackets $\langle k,\eta 	\rangle \le C \langle \eta-\xi \rangle \langle k,\xi \rangle$.
\end{proof}

Next, we use the following key lemma to study the homogeneous problem associated to \eqref{you11}.
\begin{lemma}\label{keylemma}
		Let $Z(t)$ be a solution to \eqref{you11} with $\widehat{\beta}^{in}+\widehat{\Gamma}^{in}=0$. Define
		\begin{align*}
			a(t)=\frac14 \frac{\dt p}{p}, \qquad b(t)=\frac{\sqrt{p}}{M}, \qquad d(t)=\frac{2\sqrt{p}}{M}+\frac{2Mk^2}{p^{3/2}}.
		\end{align*}
		and
		\begin{align*}		E(t)=\left(\sqrt{\frac{d}{b}}|Z_1|^2\right)(t)+\left(\sqrt{\frac{b}{d}}|Z_2|^2\right)(t)+2\left(\frac{a}{\sqrt{db}}
Re(Z_1\bar{Z}_2)\right)(t).
		\end{align*}
		Then, there exists constants $c_1,C_1,c_2,C_2>0$ independent of $k,\eta$ such that
		\begin{align*}
			c_1E(0)\leq E(t)\leq C_1 E(0),
		\end{align*}
		and
		\begin{align*}
			c_2|Z^{in}|\leq \big|Z(t)|\leq C_2|Z^{in}|.
		\end{align*}
		
	\end{lemma}
\begin{proof}
	The proof of this lemma can be founded in \cite{dole}, \cite{doln}, here, we omit the details.
\end{proof}

To prove our theorem, we must use the  integrated form of the solutions to \eqref{you11}. That is to say, we can write the solutions of \eqref{you11} by
\begin{equation}\label{ping1}	
		Z(t)=\Phi_L(t,0)\bigg(Z^{in}+\int_{0}^{t}\Phi_L(0,s)F(s)(\widehat{\beta}^{in}+\widehat{\Gamma}^{in})\,ds\bigg),
	\end{equation}
	where $\Phi_L$ is the solution operator  related to the  equation
$${\frac{{\mathrm{d}}}{{\mathrm{d}} t}} Z(t)=L(t)Z(t).$$

According to Lemma \ref{keylemma}, there holds
\begin{equation}\label{ping2}	
			\begin{split}
				\int_0^\infty |\Phi_L(t,s)F(s)|\,ds&\le C \int_0^\infty|F(s)|\,ds\\
				&\le  \frac{C}{|k|^{\frac32}}\int_{0}^{\infty}\frac{ds }{(1+(\eta/k-s)^2)^{\frac74}}\\
&\le \frac{C}{\gamma |k|^{\frac32}}.
			\end{split}
		\end{equation}
Hence, for any $t\geq0$, using Lemma \ref{keylemma} once again, we deduce from \eqref{ping1}	 and \eqref{ping2} that
\begin{equation}\label{ping3}
			|\widehat{Z}(t,k,\eta)|\lesssim |Z^{in}(k,\eta)|+\frac{1}{\gamma}|\widehat{\beta}^{in}(k,\eta)+\widehat{\Gamma}^{in}(k,\eta)| .
		\end{equation}

Now, we begin to give the details of the estimates \eqref{you16}--\eqref{you18}.
We first deduce from \eqref{you5}, \eqref{you6} that
\begin{align*}
\Omega(t,X,Y)=\frac{{\beta}^{in}(X,Y)+\Ga^{in}(X,Y)}{\gamma}-\delta
\end{align*}
which implies
		\begin{equation}
			|{\widehat{\Omega}}|(t,k,\eta)\leq|{\widehat{\delta}}|(t,k,\eta)+\left|\frac{\widehat{\beta}^{in}+\widehat{\Ga}^{in}}{\gamma}\right|(k,\eta).
		\end{equation}
As a result, it follows from \eqref{m7} that
		\begin{align*}
			\norma{\p[{\mathbf{v}}]^x(t)}{L^2}^2&=\norma{(\dy \Delta^{-1}\omega)(t)}{L^2}^2
			=\sum_k\int \frac{(\eta-kt)^2}{p^{2}}|{\widehat{\Omega}}(t)|^2\,d\eta\\
			&\le C \sum_k\int\Big( {M}^2\frac{(\eta-kt)^2}{p^{3/2}}\left|\frac{{\widehat{\delta}}(t)}{{M} p^{1/4}}\right|^2+\frac{(\eta-kt)^2}{p^2}\left|\frac{\widehat{\beta}^{in}+\widehat{\Ga}^{in}}{\gamma}\right|^2\Big)\,d\eta\\
&\le C \sum_k\int\Big( {M}^2\frac{(\eta-kt)^2}{p^{3/2}}|\widehat{Z}|^2+\frac{(\eta-kt)^2}{p^2}
\left|\frac{\widehat{\beta}^{in}+\widehat{\Ga}^{in}}{\gamma}\right|^2\Big)\,d\eta.
		\end{align*}
	In view of \eqref{ping3}, Lemma \ref{you14} and the definition in \eqref{you3} and	\eqref{you13}, we further get
		\begin{align}\label{ddg}
			\norma{\p[{\mathbf{v}}]^x(t)}{L^2}^2
\le& C\sum_k\int\Big( \frac{{M}^2}{\sqrt{p}}(|Z^{in}|^2+\left|\frac{\widehat{\beta}^{in}+\widehat{\Ga}^{in}}{\gamma}\right|^2)+\frac{1}{p}
\left|\frac{\widehat{\beta}^{in}+\widehat{\Ga}^{in}}{\gamma}\right|^2\Big)\,d\eta\nn\\
\le &C\frac{{M}^2}{\langle t\rangle}\bigg(\norma{ \frac{\rho^{in}+\ta^{in}}{\gamma{M}}}{H^{-\frac12}_xL^2_y}^2+\norma{\alpha^{in}}{H^{-\frac12}_xH^{-1}_y}^2 +\norma{\frac{\rho^{in}+\ta^{in}+\gamma\omega^{in}}{\gamma}}{H^{-\frac12}_xH^{\frac12}_y}^2\bigg)\nn\\
			&+\frac{C}{\langle t\rangle^2}\norma{\frac{\rho^{in}+\ta^{in}+\gamma\omega^{in}}{\gamma}}{H^{-1}_xH^1_y}^2.
		\end{align}
In a similar manner, we deal with
	$\p[{\mathbf{v}}]^y$  as follows
		\begin{align}
			\norma{\p[{\mathbf{v}}]^y(t)}{L^2}^2=&\norma{\dx \Delta^{-1}\omega}{L^2}^2 \nn\\
			\le & C\sum_k\int \Big({M}^2\frac{k^2}{p^{3/2}}\left|\frac{{\widehat{\delta}}(t)}{{M} p^{1/4}}\right|^2+\frac{k^2}{p^2}\left|\frac{\widehat{\beta}^{in}+\widehat{\Ga}^{in}}{\gamma}\right|^2\Big)\,d\eta\nn\\
			\le &C\frac{{M}^2}{\langle t\rangle^3}\bigg(\norma{ \frac{\rho^{in}+\ta^{in}}{\gamma{M}}}{H^{-\frac12}_xH^{1}_y}^2
+\norma{\alpha^{in}}{H^{-\frac12}_xL^2_y}^2+\norma{\frac{\rho^{in}
+\ta^{in}+\gamma\omega^{in}}{\gamma}}{H^{-\frac12}_xH^{\frac{3}{2}}_y}^2\bigg)\nn\\
			&+\frac{C}{\langle t\rangle^4}\norma{\frac{\rho^{in}+\ta^{in}+\gamma\omega^{in}}{\gamma}}{H^{-1}_xH^2_y}^2.
		\end{align}

Finally, we have to  bound the compressible part of the velocity, the density and the temperature.

On the one hand, we infer from	 the Helmholtz decomposition that
		\begin{align*}
			&\norma{Q[{\mathbf{v}}](t)}{L^2}^2+\frac{1}{{M}^2}\norma{\rho(t)+\ta(t)}{L^2}^2\\
&\quad=\norma{(\dx \Delta^{-1}\alpha)(t)}{L^2}^2+\norma{(\dy \Delta^{-1}\alpha)(t)}{L^2}^2+\frac{1}{{M}^2}\norma{\rho(t)+\ta(t)}{L^2}^2\\
			&\quad=\sum_k\int \left( \frac{|{\widehat{\alpha}}(t)|^2(t,k,\eta)}{k^2+\eta^2}+\frac{1}{{M}^2}|{\widehat{\rho}}(t)+\widehat{\ta}(t)|^2(t,k,\eta)\right)\,d\eta\\
			&\quad=\sum_k\int\left( \frac{|{\widehat{A}}(t)|^2}{p}(t,k,\eta)+\frac{1}{{M}^2}|{\widehat{R}}(t)+\widehat{{\Theta}}(t)|^2(t,k,\eta)\right)\,d\eta\\
&\quad=\sum_k\int\left( \frac{|{\widehat{A}}(t)|^2}{p}(t,k,\eta)+\frac{\gamma^2}{{M}^2}|{\widehat{\delta}}(t)|^2(t,k,\eta)\right)\,d\eta.
		\end{align*}
As a result, we can further deduce from
 \eqref{ping3}, Lemma \ref{you14} and the fact that $p\leq\jap{t}^2\jap{k,\eta}^2$ that
		\begin{align}\label{ai}
&\norma{Q[{\mathbf{v}}](t)}{L^2}^2+\frac{1}{{M}^2}\norma{\rho(t)+\ta(t)}{L^2}^2
\nn\\
&\quad=  \sum_k\int \sqrt{p}\bigg(\bigg|\frac{{\widehat{A}}(t)}{p^{3/4}}\bigg|^2+\bigg|\frac{{\widehat{\delta}}(t)}{{M} p^{1/4}}\bigg|^2\bigg)\,d\eta\nn\\
			&\quad= \sum_k \int \sqrt{p}|\widehat{Z}(t)|^2\,d\eta\nn\\
			&\quad\le C \jap{t}\left(\norm{Z^{in}}_{H^{\frac12}}^2+\norm{\frac{\rho^{in}+\ta^{in}+\gamma\omega^{in}}{\gamma}}_{H^{\frac12}}^2\right)
\nn\\
			&\quad\le C 	\jap{t}\left(\norm{\frac{\rho^{in}+\ta^{in}}{\gamma M}}_{L^2}^2+\norm{\alpha^{in}}^2_{H^{-1}}
+\norm{\frac{\rho^{in}+\ta^{in}+\gamma\omega^{in}}{\gamma}}_{H^{\frac12}}^2\right).
		\end{align}
		
On the other hand, by \eqref{m10}, there holds
\begin{align*}
 &(\dt+y\dx) ((\gamma-1)\rho-\ta)=0
		\end{align*}
which implies that
\begin{align*}
(\gamma-1)\rho-\ta=(\gamma-1)\rho^{in}-\ta^{in}.
		\end{align*}
Moreover, we have
\begin{align*}
\left\|\frac{(\gamma-1)\rho(t)-\ta(t)}{{M}}\right\|_{{L^2}}^2
=\left\|\frac{(\gamma-1)\rho^{in}-\ta^{in}}{{M}}\right\|_{{L^2}}^2.
\end{align*}

Let
\begin{eqnarray*}
\left\{\begin{aligned}
&x_1=\frac{(\gamma-1)\rho-\ta}{{M}},\\
&y_1=\frac{\rho+\ta}{{M}},
\end{aligned}\right.
\end{eqnarray*}
which gives
\begin{eqnarray*}
\left\{\begin{aligned}
&\frac{\gamma}{{M}}\rho={x_1+y_1},\\
&\frac{\gamma}{{M}}\ta={(\gamma-1)}x_1-y_1.
\end{aligned}\right.
\end{eqnarray*}
Hence
\begin{align*}
\frac{\gamma^2}{{M}^2}\norma{\rho(t)}{L^2}^2
=& \left(\left\|\frac{(\gamma-1)\rho-\ta}{{M}}\right\|_{L^2}^2+\left\|\frac{\rho+\ta}{{M}}\right\|_{L^2}^2\right)\nonumber\\
\le& C\left\|\frac{(\gamma-1)\rho^{in}-\ta^{in}}{{M}}\right\|_{{L^2}}^2
\nonumber\\
&
+  C\jap{t}\left(\norm{\frac{\rho^{in}+\ta^{in}}{\gamma M}}_{L^2}^2+\norm{\alpha^{in}}^2_{H^{-1}}
+\norm{\frac{\rho^{in}+\ta^{in}+\gamma\omega^{in}}{\gamma}}_{H^{\frac12}}^2\right),
\end{align*}
and
\begin{align*}
\frac{\gamma^2}{{M}^2}\norma{\ta(t)}{L^2}^2
=&{(\gamma-1)^2} \left\|\frac{(\gamma-1)\rho-\ta}{{M}}\right\|_{L^2}^2+\left\|\frac{\rho+\ta}{{M}}\right\|_{L^2}^2\nonumber\\
\le& {C(\gamma-1)^2}\left\|\frac{(\gamma-1)\rho^{in}-\ta^{in}}{{M}}\right\|_{{L^2}}^2
\nonumber\\
&
+  C\jap{t}\left(\norm{\frac{\rho^{in}+\ta^{in}}{\gamma M}}_{L^2}^2+\norm{\alpha^{in}}^2_{H^{-1}}
+\norm{\frac{\rho^{in}+\ta^{in}+\gamma\omega^{in}}{\gamma}}_{H^{\frac12}}^2\right).
\end{align*}

Consequently, we have completed the proof  of the Theorem \ref{dingli}.

\bigskip

 \textbf{Acknowledgement.}
This work is supported by  NSFC under grant number 11601533.

\end{document}